\documentclass[11pt,reqno]{amsart}
\usepackage{amsmath, amsthm, amscd, amsfonts, amssymb, graphicx, color}
\usepackage[colorlinks]{hyperref}

 \makeatletter
 \evensidemargin  \oddsidemargin
 \makeatother

\begin{document}

\setcounter{page}{1}

\begin{center}
{\large\bf A Characterization of Metric Projection in CAT(0) Spaces}

\vskip.25in

{\bf Hossein Dehghan } \\
{\footnotesize \textit{Department of Mathematics, Institute for
Advanced Studies in Basic Sciences (IASBS), Gava Zang, Zanjan 45137-66731, Iran}\\[-1mm]
\textit{E-mail:} {\tt hossein.dehgan@gmail.com}}\\[2mm]

{\bf Jamal Rooin} \\
{\footnotesize \textit{Department of Mathematics, Institute for Advanced Studies in Basic Sciences (IASBS), Gava Zang, Zanjan 45137-66731, Iran}\\[-1mm]
\textit{E-mail:} {\tt rooin@iasbs.ac.ir}}\\[2mm]

\end{center}

\vskip 5mm

\noindent{\footnotesize{\bf Abstract.}
In this paper, we present a characterization of metric projection in CAT(0) spaces by using the concept of quasilinearization. Furthermore, some basic properties of matric projection are investigated.  \\



   \newtheorem{df}{Definition}[section]
   \newtheorem{rk}[df]{Remark}
   \newtheorem{lem}[df]{Lemma}
   \newtheorem{thm}[df]{Theorem}
   \newtheorem{pro}[df]{Proposition}
   \newtheorem{cor}[df]{Corollary}
   \newtheorem{ex}[df]{Example}

 \setcounter{section}{0}
 \numberwithin{equation}{section}

\vskip .2in

%
%
\section{Introduction}

\par
A metric space $(X,d)$ is a CAT(0) space if it is geodesically connected and if every geodesic triangle in $X$ is at least as thin as its
comparison triangle in the Euclidean plane. For other equivalent definitions and basic properties, we refer the reader to standard texts such as \cite{Bridson}.
Complete CAT(0) spaces are often called Hadamard spaces.
Let $x,y\in X$. We write $\lambda x\oplus (1-\lambda) y$ for the the unique point $z$ in the geodesic segment joining from $x$ to $y$ such that
\begin{eqnarray*}\label{oplus}
d(z, x)= (1-\lambda)d(x, y)\ \ \ \mbox{and}\ \ \  d(z, y)=\lambda d(x, y).
\end{eqnarray*}
We also denote by $[x, y]$ the geodesic segment joining from $x$ to $y$, that is, $[x, y]=\{\lambda x\oplus (1-\lambda) y : \lambda\in [0, 1]\}$. A subset $C$ of a CAT(0) space is convex if $[x, y]\subseteq C$ for all $x, y\in C$.
\par
Berg and Nikolaev in \cite{Berg1} have introduced the concept of \emph{quasilinearization}. Let us formally denote a pair $(a,b)\in X\times X$ by $\overrightarrow{ab}$ and call it a vector. Then quasilinearization is the map $\langle \cdot, \cdot \rangle : (X\times X)\times(X\times X)\to \mathbb{R}$ defined by
\begin{eqnarray}\label{qasilin}
 \langle\overrightarrow{ab},\overrightarrow{cd}\rangle = \frac{1}{2}\left(d^2(a,d)+d^2(b,c)-d^2(a,c)-d^2(b,d)\right), \ \ \ \ (a,b,c,d\in X).
\end{eqnarray}
We say that $X$ satisfies the Cauchy-Schwarz inequality if
\begin{eqnarray}\label{Cha-Sch}
\langle\overrightarrow{ab},\overrightarrow{cd}\rangle\leq d(a,b) d(c,d)
\end{eqnarray}
for all $a,b,c,d\in X$. It known \cite[Corollary 3]{Berg1} that a geodesically connected metric space is CAT(0) space if and only if it satisfies the Cauchy-Schwarz inequality.
\par
We need the following lemma in the sequel.

\begin{lem}\cite[Lemma 2.5]{Dhompongsa}
A geodesic space $X$ is a CAT(0) space if and only if the following inequality
\begin{eqnarray}\label{para ine}
 d^2(\lambda x\oplus (1-\lambda) y, z)\leq \lambda d^2(x,z)+(1-\lambda)d^2(y,z)-\lambda(1-\lambda)d^2(x,y)
\end{eqnarray}
is satisfied for all $x,y,z\in X$ and $\lambda\in [0,1]$.
\end{lem}

\section{Main results}
Let $C$ be a nonempty complete convex subset of a CAT(0) space $X$. It is known \cite[Proposition 2.4]{Bridson} that for any $x\in X$ there exists a unique point
$x_0\in C$ such that
\begin{eqnarray}
\nonumber d(x,x_0)= \min_{ y\in C} d(x, y).
\end{eqnarray}
The mapping $P_C : X \rightarrow C$ defined by $P_Cx= x_0$ is called the \emph{metric projection} from $X$ onto $C$.
 \par
 We need the following useful lemma to prove our main result.

\begin{lem}\label{coffi quasi lem} (For a general case see \cite[Lemma 4.1.1]{Dehghan})
Let $X$ be a CAT(0) space, $x,y\in X$, $\lambda\in [0,1]$ and $z= \lambda x\oplus (1-\lambda)y$. Then,
\begin{eqnarray}\label{coffi quasi}
\langle\overrightarrow{zy},\overrightarrow{zw}\rangle \leq \lambda \langle\overrightarrow{xy},\overrightarrow{zw}\rangle
\end{eqnarray}
for all $w\in X$.
\end{lem}
\begin{proof}
 Using (\ref{oplus}) and (\ref{para ine}), we have
\begin{eqnarray}
\nonumber 2(\langle\overrightarrow{zy},\overrightarrow{zw}\rangle-\lambda \langle\overrightarrow{xy},\overrightarrow{zw}\rangle)&=& d^2(z,w) + d^2(y,z) - d^2(y,w)\\
\nonumber && -\lambda ( d^2(x,w) + d^2(y,z) - d^2(x,z) - d^2(y,w)  )\\
\nonumber&\leq& \lambda d^2(x,w) + (1-\lambda) d^2(y,w) -\lambda(1- \lambda) d^2(x,y)+  d^2(y,z) \\
\nonumber && - d^2(y,w) -\lambda ( d^2(x,w) + d^2(y,z) - d^2(x,z) - d^2(y,w)  )\\
\nonumber &=& (1-\lambda) d^2(y,z) + \lambda d^2(x,z) -\lambda(1- \lambda) d^2(x,y)\\
\nonumber &=& \lambda^2 (1-\lambda) d^2(y,x) + \lambda (1-\lambda)^2 d^2(x,y) -\lambda(1- \lambda) d^2(x,y)\\
\nonumber &=& 0,
\end{eqnarray}
which is the desired inequality.
\end{proof}


\begin{thm}\label{proj}
Let $C$ be a nonempty convex subset of a CAT(0) space $X$, $x\in X$ and $u\in C$. Then  $u=P_Cx$ if and only if
\begin{eqnarray}\label{charac}
\langle\overrightarrow{xu},\overrightarrow{uy}\rangle\geq0
\end{eqnarray}
 for all $y\in C$.
\end{thm}

\begin{proof}
Let $\langle\overrightarrow{xu},\overrightarrow{uy}\rangle\geq0$  for all $y\in C$. If $d(x,u)=0$, then the assertion is clear. Otherwise, we have
\begin{eqnarray}
\nonumber
\langle\overrightarrow{xu},\overrightarrow{xy}\rangle -\langle\overrightarrow{xu},\overrightarrow{xu}\rangle= \langle\overrightarrow{xu},\overrightarrow{uy}\rangle\geq0.
\end{eqnarray}
This together with Cauchy-Schwarz inequality implies that
\begin{eqnarray}
\nonumber d^2(x,u)=\langle\overrightarrow{xu},\overrightarrow{xu}\rangle \leq \langle\overrightarrow{xu},\overrightarrow{xy}\rangle
\leq d(x,u) d(x,y) .
\end{eqnarray}
That is, $d(x,u)\leq d(x,y)$ for all $y\in C$ and so $u=P_Cx$.
\par
Conversely, let $u=P_Cx$. Since $C$ is convex, then $z=\lambda y\oplus(1-\lambda)u\in C$ for all $y\in C$ and $\lambda\in(0,1)$. Thus, $d(x,u)\leq d(x,z)$. Using (\ref{qasilin}) we have
\begin{eqnarray}\label{char ineq1}
 \langle\overrightarrow{xz},\overrightarrow{uz}\rangle \geq \frac{1}{2}d^2(x,z)-\frac{1}{2}d^2(x,u)\geq 0.
\end{eqnarray}
On the other hand, by using Lemma \ref{coffi quasi lem}, we have $\langle\overrightarrow{xz},\overrightarrow{uz}\rangle\leq \lambda\langle\overrightarrow{xz},\overrightarrow{uy}\rangle$. This together with (\ref{char ineq1}) implies that
\begin{eqnarray}
\nonumber \langle\overrightarrow{xz},\overrightarrow{uy}\rangle \geq 0.
\end{eqnarray}
Since the function $d(\cdot , x): X \to \mathbb{R}$ is continuous for all $x\in X$, letting $\lambda\to 0^+$, we have $\langle\overrightarrow{xu},\overrightarrow{uy}\rangle\geq0$. This completes the proof.
\end{proof}

\begin{thm}
Let $C$ be a nonempty subset of a CAT(0) space $X$ and $x\in X$. Then $P_Cx\subset \partial C$, where 
$ P_Cx=\{z\in C: d(x,z)= \inf_{ y\in C} d(x, y)\}$ and 
$\partial C$ is the boundary of $C$.
\end{thm}
\begin{proof}
Let $u\in P_Cx$ and $u\not\in \partial C$. Then there exists an $\varepsilon> 0$ such that $B(u,\varepsilon)\subset C$, where $B(u,\varepsilon)$ denotes the open ball with center $u$ and radius $\varepsilon$. For each $n\geq 1$, let $z_n=1/n x\oplus (1-1/n)u$. We know that
\begin{eqnarray}
\nonumber d(z_n,u)= \frac{1}{n}d(x,u).
\end{eqnarray}
Hence, For sufficiently large $N\geq 1$, $d(z_N,u)<\varepsilon$. Thus $z_N\in B(u,\varepsilon)\subset C$. On the other hand,
\begin{eqnarray}
\nonumber d(z_N,x)=\left(1- \frac{1}{N}\right)d(x,u)<d(x,u)= d(x,C),
\end{eqnarray}
which contradicts the fact that $u\in P_Cx$. Therefore, $u\in \partial C$.
\end{proof}

\par
A self-mapping $T$ of $C\subseteq X$ is said to be
\begin{itemize}
  \item[(i)] \emph{nonexpansive} if $d( Tx, Ty)\leq d(x,y)$,
  \item[(ii)] \emph{firmly nonexpansive} if $\langle\overrightarrow{xy},\overrightarrow{Tx Ty}\rangle\geq d^2(Tx, Ty)$,
  \item[(iii)] \emph{monotone} if $\langle\overrightarrow{xy},\overrightarrow{Tx Ty}\rangle\geq 0$,
\end{itemize}
for all $x,y\in C$. It is clear that every firmly nonexpansive mapping is monotone. Also, it follows from Cauchy-Schwarz inequality that every firmly nonexpansive mapping is nonexpansive.
\begin{pro}\label{firmly nonex}
Let $C$ be a nonempty closed convex subset of a Hadamard space $X$. Then, the metric projection $P_C : X \rightarrow C\subseteq X$ is firmly nonexpansive and so it is monotone and nonexpansive.
\end{pro}
\begin{proof}
 Let $x,y\in X$. Since $P_Cx, P_Cy\in C$, it follows from Theorem \ref{proj} that
\begin{eqnarray}
\nonumber \langle\overrightarrow{xP_Cx},\overrightarrow{P_Cx P_Cy}\rangle\geq 0\ \ \  \ \mbox{and}\ \ \
\langle\overrightarrow{yP_Cy},\overrightarrow{P_Cy P_Cx}\rangle\geq0.
\end{eqnarray}
Therefore,
\begin{eqnarray}\label{P_C firmly nonexp}
\nonumber \langle\overrightarrow{xy},\overrightarrow{P_Cx P_Cy}\rangle&=&\langle\overrightarrow{xP_Cx},\overrightarrow{P_Cx P_Cy}\rangle+\langle\overrightarrow{P_Cx P_Cy},\overrightarrow{P_Cx P_Cy}\rangle+\langle\overrightarrow{yP_Cy},\overrightarrow{P_Cy P_Cx}\rangle\\
\nonumber &\geq&\langle\overrightarrow{P_Cx P_Cy},\overrightarrow{P_Cx P_Cy}\rangle\\
\nonumber &=&d^2(P_Cx, P_Cy),
\end{eqnarray}
which completes the proof.
\end{proof}

{\small

}
\end{document}